\documentclass{amsart}
\usepackage{amsmath,amsfonts,euscript,amssymb,amsthm,graphicx,subfigure,verbatim}
\usepackage{graphicx}
\usepackage{hyperref}
\usepackage{enumerate}
\usepackage{color}      

\newtheorem{theorem}{Theorem}

\newtheorem{lemma}{Lemma}
\newtheorem{proposition}{Proposition}

\numberwithin{equation}{section}

\newcommand{\R}{\mathbb{R}}
\newcommand{\N}{\mathbb{N}}
\newcommand{\Z}{\mathbb{Z}}
\newcommand{\C}{\mathbb{C}}

\newcommand{\St}{\mathbb{S}^2}

\newcommand{\scp}[2]{\left\langle #1, #2 \right\rangle}

\newcommand{\m}{\boldsymbol{m}}

\newcommand{\eps}{\varepsilon}
\newcommand{\del}{\partial}
\newcommand{\ein}{\boldsymbol{\hat{e}}}

\newcommand{\bs}[1]{ \boldsymbol{#1}}
\newcommand{\dif}{\ensuremath{\,\mathrm{d}}}

\renewcommand{\div}{\mathop{\mathrm{div}}\nolimits}

\newcommand{\ra}{\rangle}
\newcommand{\la}{\langle}

\def\XXint#1#2#3{{\setbox0=\hbox{$#1{#2#3}{\int}$}
\vcenter{\hbox{$#2#3$}}\kern-.5\wd0}}

\title{Stability of axisymmetric chiral skyrmions}

\author{Xinye Li}
\address{RWTH Aachen University\\Lehrstuhl I f\"ur Mathematik\\Pontdriesch 14-16\\52056 Aachen}
\email{xli@math1.rwth-aachen.de}

\author{Christof Melcher}
\address{RWTH Aachen University\\Lehrstuhl I f\"ur Mathematik and JARA -- Fundamentals of Future Information Technology\\Pontdriesch 14-16\\52056 Aachen}
\email{melcher@rwth-aachen.de}

\begin{document}

\begin{abstract}
We examine topological solitons in a minimal variational model for a 
chiral magnet, so-called chiral skyrmions. In the regime of large background fields, we prove linear stability of axisymmetric chiral skyrmions under arbitrary perturbations in the energy space, a long-standing open question in physics literature. Moreover, we show strict local minimality of axisymmetric chiral skyrmions and
nearby existence of moving soliton solution for the Landau-Lifshitz-Gilbert equation driven by a small spin transfer torque.
\end{abstract}
\maketitle

\section{Introduction and main results}

We are concerned with topological solitons $\m:\R^2  \to \St$ occurring in two-dimensional chiral magnets
governed by interaction energies of the form
\begin{equation}\label{eq:energy}
E(\m)= \int_{\R^2} \frac{1}{2} |\nabla \m|^2 +\m \cdot (\nabla \times \m) + \frac{h}{2}|\m-\ein_3|^2 \dif x.
\end{equation}
The hallmark of such systems is the helicity term $\m \cdot (\nabla \times \m)$ arising from antisymmetric exchange also known as Dzyaloshinkii-Moriya interaction. Helicity suspends the Derrick-Pohozaev non-existence criterion and breaks independent $O(2)$ invariance in the target and the domain. It enables the reduction of energy by twisting,
an effect that is maximized if the horizontal magnetization field $m=(m_1,m_2)$ is antiparallel to the level sets of $m_3$, see \cite{melcher2014chiral}.
Isolated chiral skyrmions are stable critical points of $E$ in a non-trivial homotopy class, characterized by the topological charge or degree
\begin{equation}\label{eq:Q}
Q(\m)=\frac{1}{4\pi} \int_{\R^2} \m \cdot \del _1 \m  \times \del_2 \m \dif x.
\end{equation}
The lack of $O(2)$ invariance leads to a specific energetic selection of degree $Q=-1$ in accordance with the twisting behaviour relative
to the selected background state $\m(\infty)=\ein_3$. 
Restricting the energy to the class of axisymmetric configurations of the form
\begin{equation}\label{eq:m0}
\m_0(r e^{\rm{i} \psi}) = \left(e^{\rm{i}\left(\psi+\frac \pi 2\right)} \sin \theta(r), \cos \theta(r) \right) 
\end{equation}
with the usual identification $\R^2 \cong \C$, the corresponding minimizing problem gives rise to a boundary value problem for a second order ordinary differential equation for the polar profile $\theta:[0,\infty) \to \R$
\begin{equation}\label{eq:thetaODE}
\theta''+\frac{1}{r}\theta'-\frac{1}{r^2} \sin \theta \cos \theta + \frac{2}{r} \sin^2 \theta - h \sin \theta = 0
\end{equation}
with boundary conditions
\begin{equation}\label{eq:thetaBoundary}
\theta(0)=\pi \quad\text{and}\quad \lim_{ r \to \infty} \theta(r) = 0.
\end{equation}
The equation has been studied extensively in physics literature mainly based on numerical simulation, exploring the occurrence and radial stability of isolated chiral skyrmions,
see \cite{bogdanov1994thermodynamically, bogdanov1999stability, leonov2016properties}. Our main result provides a rigorous confirmation and shows the linear stability of $\m_0$ with respect to arbitrary perturbations in the energy space if $h$ is sufficiently large.
\begin{theorem}\label{thm:stability}
There exists a unique solution to the boundary value problem \eqref{eq:thetaODE}-\eqref{eq:thetaBoundary} provided $h \gg 1$. 
The field $\m_0$ given by \eqref{eq:m0} is a linearly stable critical point of (\ref{eq:energy}). More precisely, the Hessian $\mathcal{H}=\mathcal{H}(\bs \phi, \bs \psi)$ at $\m_0$ satisfies for some $\lambda>0$
\begin{eqnarray*}
\mathcal{H}(\bs \phi,\bs \phi) &\geq& \lambda \| \bs \phi - \bs \phi_0 \|_{H^1}^2 \quad \text{for all tangent fields  
$\bs \phi \in H^1$ along } \m_0,
\end{eqnarray*}
where $ \bs \phi_0$ is the $L^2$ projection of $\bs \phi$ onto 
$\mathrm{span}\{ \del_1 \m_0, \del_2 \m_0\}$, the kernel of the associated Jacobi operator, which is a
Fredholm operator of index zero.
\end{theorem}
The existence result is obtained by a variational method using ingredients from \cite{melcher2014chiral}. Uniqueness follows from some fine analysis of the boundary value problem (\ref{eq:thetaODE})-(\ref{eq:thetaBoundary}) of the polar profile $\theta$, which can be found in the appendix. The main result about the stability of $\m_0$ is obtained by an adaption of the Fourier argument in \cite{mironescu1995stability, ignat2015stability, ignat2016stability}. Representing $\bs \phi$ in terms of a moving frame along $T_{\m_0} \St$, we use a Fourier expansion to eliminate the dependence on the angular coordinate and thereby the Hessian becomes a series depending only on the radial coordinate. Then we show that non-negativity depends only on the first two modes, and we prove their non-negativity by means of a decomposition argument. The spectral gap and Fredholm property are then proven by means of a variational argument. The latter one can be used to prove strict local minimality modulo translations:

\begin{theorem}\label{thm:minimality}
For $h \gg 1$ there exist $\mu >0$ and $\eps>0$ such that 
\[
E(\m) - E(\m_0) \ge \mu \inf_{x \in \R^2} \|\m(\cdot -x) - \m_0\|_{H^1}^2 
\]
for all $\m \in H^1(\R^2; \St)$ such that $\|\m-\m_0\|_{H^1}< \eps$. 
\end{theorem}

A major open question is whether $\m_0$ minimizes the energy even globally within the homotopy class $Q=-1$. Energy minimizing chiral skyrmions have been constructed \cite{melcher2014chiral} by means of a concentration-compactness argument, which extend to the case where helicity is replaced by a general form of anisotropic Dyaloshinskii-Moriya interaction \cite{hoffmann2017antiskyrmions}, where axisymmetry is lost.  Notably while Theorem \ref{thm:minimality} provides the strict local minimality for axisymmetric chiral skyrmions, the analog stability statement is not at hand for minimizing chiral skyrmions.

Starting from linear stability of static solitons and a perturbation we shall also consider the slow solitonic motion of chiral skyrmions driven by horizontal currents $v \in \R^2$ of small size. The Landau-Lifshitz-Gilbert equation with adiabatic and non-adiabatic spin transfer torques is given by 
\begin{equation}\label{eq:LLG}
\del_t \m + (v \cdot \nabla) \m = \m \times \bigg[ \alpha \del_t \m + \beta(v \cdot \nabla) \m - \bs h_{\rm eff}(\m) \bigg],
\end{equation}
where the effective field $\bs h_{\rm eff}(\m)$ is minus the $L^2$ gradient of $E$, and $\alpha, \beta>0$ are the Gilbert damping factor and the ratio of non-adiabaticity, 
respectively, see e.g. \cite{melcher2013landau, komineas2015skyrmion}. We shall construct moving soliton solutions $\m= \m(x-ct)$ with propagation velocities $c \in \R^2$.
In the case $\alpha=\beta$, such solutions are obtained from stationary solutions with $c=v$, i.e., skyrmions moves parallel to the current without changing shape. For $\alpha \not= \beta$, however, skyrmions deform and are deflected according to the skyrmion hall effect with a propagation speed $c$ implicitly determined by Thiele's equation 
\begin{equation} \label{eq:Thiele}
4\pi (v-c)^{\perp} = \mathcal{D}(\beta v - \alpha c) 
\end{equation}
where the dissipative tensor $\mathcal{D} \in \R^{2 \times 2}$ is given by
\[
\mathcal{D}_{jk} = \int_{\R^2} \del_j \m \cdot \del_k \m \dif x,
\]
see e.g. \cite{kurzke2011vortex} and references therein.

\begin{theorem}\label{thm:wave} For $h \gg 1$ there exists $\eps>0$ such that  (\ref{eq:LLG}) possesses for $|v| < \eps$  a unique family of moving soliton solution $\m=\m(x-ct)$ in an $H^2$ neighborhood of $\m_0$ with $c$ determined by \eqref{eq:Thiele}. 
\end{theorem}

The proof follows from an adaption of the continuation argument in \cite{capella2007wave}, using the implicit function theorem on Hilbert manifolds with the spin velocity $v \in \R^2$ serving as a perturbation parameter. Ingredients are in particular the non-degeneracy and Fredholm property of the linearized equilibrium equation as a consequence of Theorem \ref{thm:stability}.

Moving soliton solutions have been observed numerically in a long time asymptotics in \cite{komineas2015skyrmion}. The dynamic stability and compactness of spin-transfer torque driven chiral skyrmions governed by a slightly modified energy functional has been  discussed in \cite{doering2017} in the context of an almost conformal regime using regularity arguments. These results, however, are only shown on a finite time horizon depending on the size of $v$. Our result confirms the existence of global solutions in form of traveling waves with small velocities. 

The remainder of the paper is organized as follows. In Section \ref{sec:stab} we prove non-negativity of the Hessian and identify its kernel. Section \ref{sec:fred} is devoted to
the spectral gap and the Fredholm property, finishing the proof of Theorem \ref{thm:stability}. Sections \ref{sec:min} and \ref{sec:trav} are devoted to the proof of Theorem \ref{thm:minimality} and Theorem \ref{thm:wave}, respectively. In the Appendix, we provide some fine properties of the polar profile $\theta$ of the axisymmetric solution, which are essential to our main results.

\section{Preliminaries} \label{sec:prel}
\subsection{Function spaces}
For $k \in \N$ we shall consider fields 
\[
\m \in H^k_{\ein_3}(\R^2;\St)=\{\m:\R^2 \to \St: \m-\ein_3 \in H^k(\R^2;\R^3)\},
\]
and tangent maps along $\m$
\[
\bs \phi \in H^k(\R^2;T_{\m}\St)=\{\bs \phi \in H^k(\R^2;\R^3): \bs \phi \perp \m \text{  a.e.}\}.
\]
We also use the spaces $H^0=L^2$ and $\displaystyle{H^\infty = \bigcap_{k \in \N}H^k}$. 

For $k \ge 2$ the corresponding elements are uniformly H\"older continuous by Sobolev embedding. Moreover the spaces $H^k_{\ein_3}(\R^2;\St)$ are Hilbert manifolds.  By virtue of the immersion theorem, the projection map 
\begin{equation}\label{eq:immersion}
\pi (\m+\bs \phi):= \frac{\m +  \bs \phi}{|\m+  \bs \phi|} \in H^k_{\ein_3}(\R^2;\St).
\end{equation}
identifies a zero neighborhood in the Hilbert space $H^k(\R^2;T_{\m}\St)$, the tangent space of 
$H^k_{\ein_3}(\R^2;\St)$ at $\m$, with an open neighborhood of $\m$ in $H^k_{\ein_3}(\R^2;\St)$.\\

We denote the pointwise orthogonal projection onto $T_{\m} \St$ by
\[
P_{\m}=\bs 1 - \m \otimes \m.
\]
The following can be proven along the lines of Lemma 3 in \cite{baldes1984stability}: 
\begin{lemma} \label{lemma:density}
For $k \in \N$ and $\m \in C^\infty(\R^2;\St)$, the orthogonal projection
\[
P_{\m}: H^k(\R^2;\R^3) \to H^k(\R^2;T_{\m}\St)
\] 
is bounded. Moreover, $C^\infty_c(\R^2;T_{\m}\St)$ is dense in $H^k(\R^2;T_{\m}\St)$.
\end{lemma}

Here $C^\infty_c(\Omega;T_{\m}\St)=\{\bs \phi \in C_c^\infty(\Omega;\R^3): \bs \phi \perp \m\}$ for open sets $\Omega \subset \R^2$.

\subsection{Continuous extension of the helicity term}
Integrability of the helicity term $\m \cdot (\nabla \times \m)$ in \eqref{eq:energy} requires appropriate decay of $\m$. Starting from the dense subclass
$\{ \m:\R^2 \to \St: \m -\ein_3 \in C^\infty_c(\R^2;\R^3)\}$, the continuous extension of helicity to $H^1_{\ein_3}(\R^2;\St)$ is obtained by adding a null Lagrangian
and integration by parts, which amounts to
\begin{equation} \label{eq:helicity}
E_H(\m) = \int_{\R^2} (\m-\ein_3) \cdot (\nabla \times \m) \dif x= 2 \int_{\R^2} (m_1 \del_2 m_3 - m_2 \del_1 m_3) \dif x,
\end{equation}
see \cite{melcher2014chiral}. This way the energy $E$ given by \eqref{eq:energy} becomes a bounded quadratic form in $\m -\ein_3 \in H^1(\R^2;\R^3)$.

\subsection{First variation} Suppose $\m_0 \in H^1_{\ein_3}(\R^2;\St)$.
For $\bs \phi \in  L^\infty \cap H^1(\R^2;T_{\m_0}\St)$ and $|t|<\|\bs \phi\|_{L^\infty}^{-1}$ the variation map
\[
\m_t := \pi(\m_0 + t \bs \phi) \in H^1_{\ein_3}(\R^2;\St)
\]
is well-defined with $\dot{\m}_0=\bs \phi$. We say that $\m_0$ is a critical point of $E$ if the first variation
\begin{equation}\label{eq:1st}
\delta E(\m) \la \bs \phi \ra = \lim_{t \to 0} \frac{E(\m_t) - E(\m_0)}{t} 
\end{equation}
vanishes for arbitrary $\bs \phi \in  L^\infty \cap H^1(\R^2;T_{\m_0}\St)$. In physical terms the first variation is given by the effective field 
\begin{equation} \label{eq:eff}
\bs h_{\rm eff}(\m)= \Delta \m - 2 \nabla \times \m + h \ein_3,
\end{equation} 
i.e., minus the $L^2$ gradient of the energy 
$\delta E(\m) \la \bs \phi \ra = -  \langle \bs h_{\rm eff}(\m), \bs \phi\rangle_{L^2}$ for all admissible $\bs \phi$.
The effective field gives rise to a generalized tension field 
\[
\bs \tau(\m)= P_{\m} \bs h_{\rm eff}(\m).
\]
Critical points $\m_0 \in H^1_{\ein_3}(\R^2;\St)$ are therefore characterized by
\begin{equation}
\bs \tau(\m_0) =0 \quad \text{or equivalently} \quad \m_0 \times \bs h_{\rm eff}(\m_0)=0
\end{equation}
in the weak sense with the interpretation $\m \times \Delta \m = \nabla \cdot ( \m \times \nabla \m)$.
\begin{proposition}\label{prop:H^2}
Critical points of finite energy belong to the space $H^\infty_{\ein_3}(\R^2;\St)$.
\end{proposition}

\begin{proof}[Sketch of the proof]
The claim follows from well-established strategies for the regularity of finite energy harmonic 
maps in two dimensions \cite{Helein_book}, starting from the equation $\bs \tau(\m) =0$, which may be written as
\[
-\Delta \m + 2 \nabla \times \m + h (\m-\ein_3) = \Lambda(\m) \m 
\]
with the Lagrange multiplier
\begin{equation}\label{eq:fm0}
\Lambda(\m) = |\nabla \m|^2 + 2\m \cdot (\nabla \times \m) + h(1-m_3).
\end{equation}
As $|\nabla \m|^2 \m = \bs B : \nabla \m$ with $\div \bs B \in L^2(\R^2;\R^3)$, it follows that $\Delta  \m$
belongs to the local Hardy space (see e.g. \cite{Taylor_Tools}) modulo $L^2$ perturbations. In turn, it follows that $\Delta \m \in L^1+L^2(\R^2;\R^3)$, 
hence $\m$ is uniformly continuous. A classical argument \cite{Ladyzhenskaya_Uralceva_book} yields $\m$ is smooth. Finally a bootstrapping argument based on testing with $\Delta^k \m$ and using interpolation inequalities yields an $H^k$-bound for every $k \in \N$.
\end{proof}

\subsection{Second variation}

Suppose $\m_0$ is a critical point of finite energy and $\bs \phi, \bs \psi \in H^2(\R^2;T_{\m_0}\St)$. Then
$\m_{st}=\pi(\m_0 + s \bs \phi + t \bs \psi)$ is a smooth map into $H^2(\R^2;T_{\m_0}\St)$ 
for $s$ and $t$ sufficiently small. The Hessian or second variation of $E$ at $\m_0$ 
is given by
\begin{equation}\label{eq:2nd}
\frac{\partial^2}{\partial s \partial t}\bigg|_{s,t=0} E(\m_{st}) 
 = \delta^2E(\m) \la \bs \phi, \bs \psi \ra -\delta E(\m) \la (\bs \phi \cdot \bs \psi)  \, \m \ra,
\end{equation}
where 
\[
\delta^2E(\m) \la \bs \phi, \bs \psi \ra = \frac{\partial^2}{\partial s \partial t}\bigg|_{s,t=0} E(\m + s \bs \phi +t \bs \psi).
\]
For a fixed critical point $\m_0 \in H^2(\R^2; \St)$ we shall denote the Hessian by
\begin{equation}\label{eq:H}
\mathcal{H}(\bs \phi, \bs \psi) = \mathcal H_{\infty}(\bs \phi, \bs \psi) - \Delta \mathcal{H}(\bs \phi, \bs \psi)
\end{equation}
with
\[
\mathcal H_{\infty}(\bs \phi, \bs \psi) = \delta^2E(\m_0) \la \bs \phi, \bs \psi \ra = \int_{\R^2}\left( \nabla \bs \phi : \nabla \bs \psi + 2\bs \phi \cdot (\nabla \times \bs \psi) + h \, \bs \phi \cdot \bs \psi  \right)\dif x
\]
and with \eqref{eq:fm0}
\begin{equation}\label{eq:DeltaH}
\Delta \mathcal{H}(\bs \phi, \bs \psi) = \delta E(\m_0) \la (\bs \phi \cdot \bs \psi) \, \m_0 \ra = \int_{\R^2} \Lambda(\m_0) (\bs \phi \cdot \bs \psi) \dif x.
\end{equation}
 We shall use the notation $\mathcal{H}(\bs \phi)=\mathcal{H}(\bs \phi, \bs \phi)$ and similar for 
 $\mathcal{H}_\infty$ and $\Delta \mathcal{H}$.
 
 \medskip

 In view of Proposition \ref{prop:gap} the Hessian has a bounded extension to $H^1(\R^2;\R^3)$.
 
 \subsection{Jacobi operator}
Using Lemma \ref{lemma:density} the Hessian defines the Jacobi operator 
\[
\mathcal{J}: H^2(\R^2; T_{\m_0}\St)  \to L^2(\R^2; T_{\m_0}\St) 
\]
by
\begin{equation} \label{eq:Hesse_Jacobi}
\scp{\mathcal{J} \bs \phi}{\bs \psi}_{L^2} = \mathcal{H}( \bs \phi, \bs \psi )
\quad \text{for all} \quad \bs \phi, \bs \psi \in H^2(\R^2; T_{\m_0}\St). 
\end{equation}
More explicitly, we have $\mathcal{J}=P_{\m_0} \circ \mathcal{L}$, where
\begin{equation}
\mathcal{L}\bs \phi = -\Delta \bs \phi+ 2 \nabla \times \bs \phi + (h-\Lambda(\m_0)) \bs \phi,
\end{equation}
extending $\mathcal{J}$ to $H^2(\R^2;\R^3)$ as a uniformly elliptic operator.
The linearization of the tension field is expressed in terms of the Jacobi operator, i.e., for a differentiable one-parameter family 
$s \mapsto \m_s \in H^2_{\ein_3}(\R^2;\St)$ with tangent field $\dot{\m}_0= \bs \phi$
\begin{equation}
\mathcal{L} \bs \phi = - \frac{\dif }{\dif s} \Big|_{s=0} \bs \tau(\m_s) \quad \text{hence} \quad \mathcal{J} \bs \phi = - \frac{\dif }{\dif s} \Big|_{s=0} \left( P_{\m_0} \bs \tau(\m_s)\right). 
\end{equation}
Jacobi fields $\bs \phi$ are solutions of $\mathcal{J}\bs \phi=0$. Every smooth family $\m_t$ with $|t|<\eps$
of critical points $\bs \tau(\m_t)=0$ gives rise to a Jacobi field $\bs \phi= \dot{\m}_0$. In particular,
the translation invariance of the energy implies for every critical point $\m_0 \in H^2_{\ein_3}(\R^2;\St)$:
\begin{equation}\label{eq:kernelJacobi}
\mathcal{J}((c \cdot\nabla ) \m_0)=0 \quad \text{for every  } c \in\R^2.
\end{equation}
 Theorem \ref{thm:stability} states that for the axisymmetric skyrmion every Jacobi field is obtained in this manner.

\subsection{Axisymmetric critical points}
Letting $\ein_r = \left(\frac x r,0\right)$ one obtains for axisymmetric $\m_0$ of the form of $\eqref{eq:m0}$
\begin{equation}
\m_0 \times \bs h_{\rm eff}(\m_0) =  - \left( \theta''+\frac{1}{r}\theta'-\frac{1}{r^2} \sin \theta \cos \theta + \frac{2}{r} \sin^2 \theta - h \sin \theta \right) \ein_r.
\end{equation}
In particular, $\m_0$ is a critical point iff the corresponding polar profile $\theta$ satisfies the ordinary differential equation \eqref{eq:thetaODE}, the Euler-Lagrange 
equation of for the radial energy, see \cite{bogdanov1994thermodynamically, bogdanov1999stability, leonov2016properties}
\begin{equation} \label{eq:radial_energy}
\mathcal E(\theta) = 2\pi \int_0^{\infty}  \left(  \frac{(\theta')^2}{2}  + \frac{\sin^2 \theta}{2r^2}  +   \theta' +  \frac{\sin \theta \cos \theta}{r} + h  (1-\cos \theta) \right) r \dif r.
\end{equation}
The finite energy condition requires $\theta(r) \to 0$ (modulo $2\pi$) as $r \to \infty$. The topological constraint $Q(\m_0)=-1$ imposes $\theta(0)=\pi$, i.e. \eqref{eq:thetaBoundary}.

\medskip

Existence of axisymmetric critical points follows from a variational argument based on energy bounds obtained in \cite{melcher2014chiral}. Recall that to ensure absolute integrability, the helicity term $E_H=E_H(\m)$ has been modified by a
null-Lagrangian, i.e., has been replaced by
\[
(\m - \ein_3) \cdot( \nabla \times \m) \quad \text{or} \quad (1-\cos \theta) \left( \theta' -  \frac{\sin \theta}{r} \right), 
\]
respectively. For a minimizing sequence $\theta_k$ for $\mathcal{E}$ satisfying \eqref{eq:thetaBoundary}, one may consider the corresponding sequence $\m_k$ of axisymmetric fields.
For $h>1$ the compactness argument in \cite{melcher2014chiral} implies $H^1$ subconvergence to an axisymmetric field $\m_0$ of the form \eqref{eq:m0}. In fact, the requisite topological lower bounds, the coercivity estimate, and the upper bound in \cite{melcher2014chiral} based on an axisymmetric construction hold true. Moreover, it follows that
$\theta_k \to \theta$ in $H^1_{\rm loc}(0,\infty)$ and $\mathcal{E}(\theta) = \lim_{k \to \infty} \mathcal{E}(\theta_k)$, i.e., $\theta$ is a minimizer and satsfies \eqref{eq:thetaODE},
 Finally, 
\[
\left| \frac{\dif \cos \theta_k}{\dif r}\right| \le   \frac{(\theta_k')^2}{2}  + \frac{\sin^2 \theta_k}{2r^2} = \frac{|\nabla \m_k|^2}{2}
\]
is uniformly integrable near $r=0$, hence $\cos \theta_k$ is equicontinuous near $r=0$, thus $\theta(0)=\pi$. A similar argument implies $\theta(r) \to 0$ as $r \to \infty$, hence
\eqref{eq:thetaBoundary}.

\medskip

Our analysis relies on some properties of the polar profile $\theta$: its monotonicity and certain decay properties. The main results needed are summarized in the following proposition,
whose proof is deferred to the appendix. 
\begin{proposition}\label{thm:theta}
For $h \gg 1$, solutions of (\ref{eq:thetaODE}), (\ref{eq:thetaBoundary}) are unique, strictly decreasing, and satisfy the following estimates for all $r \in (0,\infty)$
\begin{equation}\label{eq:theta_cos}
|\cos \theta - r\sin \theta| < \frac{3}{2},
\end{equation}
\begin{equation}\label{eq:theta_sin}
r^2(\theta'(r))^2 \geq \sin^2 \theta(r),
\end{equation}
and
\begin{equation}\label{eq:theta_h}
h-\frac{3}{2r}\sin \theta \geq 0.
\end{equation}

\end{proposition}

\subsection{Axisymmetric non-existence}
Finally we note that, relative to the given background state $\m(\infty)=\ein_3$, axisymmetric chiral skyrmions are only possible for $Q =-1$.

\begin{proposition} For $N \in \Z \setminus \{0,-1\}$ there are no axisymmetric critical points
\[
\m(r e^{\rm{i} \psi})=\left( e^{\mathrm{i}(\gamma-N \psi)} \sin \theta(r), \cos \theta(r) \right)
\]
for any phase $0\le \gamma <2\pi$.
\end{proposition}
\begin{proof}
The usual Derrick-Pohozaev argument implies $E_H(\m)=-2E_Z(\m)$ for any 
critical point $\m$, where $E_H(\m)$ is given by \eqref{eq:helicity} and
\[
E_Z(\m)=\int_{\R^2} \frac{h}{2}|\m-\ein_3|^2 \dif x.
\]
is the Zeeman energy. On the other hand, a straightforward calculation yields that for an axisymmetric field with $N \not=0$
 \[
 E_H(\m) = - 2\int_0^\infty \sin^{2}\theta(r)\; \theta'(r)r \dif{r}  \int_0^{2\pi} \sin((N+1) \psi - \gamma) \dif{\psi} 
\]
vanishes unless $N = -1$. 
\end{proof}

Note that for $N=-1$ the helicity is minimized for $\gamma=\frac{\pi}{2}$.

\section{Linear stability}
\label{sec:stab}

In this section we prove non-negativity of the Hessian $ \mathcal H$ for large $h$.
\begin{proposition} \label{prop:stability}
For $h \gg 1$ the Hessian $\mathcal H$ satisfies
\begin{itemize}
\item[(i)] $\mathcal{H}(\bs \phi) \geq 0$ for all $\bs \phi \in H^1(\R^2;T_{\m_0}\St)$ 
\item[(ii)] $\mathcal{H}(\bs \phi)= 0$ if and only if $\bs \phi \in \mathrm{span}\{\del_1 \m_0, \del_2 \m_0 \}$.
\end{itemize}
\end{proposition}

\subsection{Approximation}
Proving (i) we can assume $\bs \phi \in C^{\infty}_c(\R^2 \setminus \{0\}; T_{\m_0}\St)$.
In fact, by Lemma \ref{lemma:density} and the $H^1$ continuity of the Hessian $\mathcal H$, it is sufficient to prove non-negativity for $\bs \phi \in H^{\infty}(\R^2;T_{\m_0} \St)$.  Moreover, it follows from \eqref{eq:theta_0} that
\[
\del_j  \m_0  \cdot \ein_k |_{x=0}= \frac{h}{2} \epsilon_{jk} \quad \text{for} \quad 1 \le j,k \le 2, 
\]
hence 
\begin{equation} \label{eq:non_deg}
\mathrm{span}\{ \del_1 \m_0, \del_2 \m_0\} |_{x=0}=T_{\m_0(0)} \St.
\end{equation}
Since by \eqref{eq:Hesse_Jacobi} and \eqref{eq:kernelJacobi} 
\begin{equation} \label{eq:add_kernel}
\mathcal{H}(\bs \phi + (c\cdot \nabla) \m_0, \bs \psi)=\mathcal{H}(\bs \phi, \bs \psi)  \quad \text{for all} \quad c \in \R^2, 
\end{equation}
in particular $\mathcal{H}(\bs \phi) =\mathcal{H}(\bs \phi + (c\cdot \nabla) \m_0)$, we can arrange $\Phi(0)=0$, and the following truncation lemma applies.

\begin{lemma}\label{lemma:approximation}
Suppose $\bs \phi \in H^{\infty}(\R^2;T_{\m_0} \St)$ satisfies $\bs \phi(0)=0$. Given $0<\eps<1$ there exists $\bs \phi_\eps \in C^{\infty}_c(\R^2;T_{\m_0} \St)$
such that  $\bs \phi_\eps = \bs \phi(x)$ for $\eps \le |x| \le 1/\eps$ and $\bs \phi_\eps \to \bs \phi$ in $H^1$ as $\eps \to 0$.
\end{lemma}

\begin{proof} Let $\rho \in C^\infty(\R^2;[0,1])$ with $\rho=0$ near the origin and $\rho=1$ outside $B_1$, and $\bs \phi_\eps= \rho_\eps \bs \phi$ where $\rho_\eps(x)=\rho(x/\eps)$. 
Clearly $\bs \phi_\eps \to \bs \phi$ in $L^2$ and $\nabla \bs \phi_\eps = \rho_\eps \nabla  \bs \phi +\nabla \rho_\eps \otimes \bs  \phi   \to\nabla  \bs \phi$ pointwise outside the originas $\eps \to 0$ . Moreover 
\[
|\nabla \bs \phi_\eps|^2 \le 2  |\nabla \bs \phi|^2 + 2  \left( \frac{L \omega_\eps}{\eps} \right)^2 \chi_{B_\eps}, 
\]
where $L =\sup_{x \in B_1}|\nabla \rho(x)|$ and $\omega_\eps=\sup_{x \in B_\eps}|\bs \phi(x)| \to 0$ as $\eps \to 0$. Hence $|\nabla \bs \phi_\eps|^2$ is tight and equiintegrable and 
therefore $\nabla \bs \phi_\eps \to 
\nabla \bs \phi$ in $L^2$ as $\eps \to 0$ by Vitali's convergence theorem. Truncation near infinity is standard. 
\end{proof}

%
In the following sections we shall show that $\mathcal{H}(\bs \phi) > 0$ for all $\bs \phi \in C^{\infty}_c(\R^2 \setminus \{0\};T_{\m_0}\St) \setminus \{0\}$, proving
in particular claim (i) of Proposition \ref{prop:stability}. Once we have shown that $\mathcal H$ is positive semi-definite and satisfies Cauchy-Schwarz, it follows that
\[
\mathcal{H}(\bs\phi)=0 \quad\text{if and only if}\quad \mathcal{H}(\bs\phi,\bs\psi)=0 \quad\text{for all } \bs\psi \in H^1(\R^2; T_{\m_0}\St).
\]
Thus if $\mathcal{H}(\bs \phi)=0$ for some $\bs \phi \in H^1(\R^2;T_{\m_0} \St)$, Lemma \ref{lemma:H^2} below and a bootstrapping argument imply
$\bs \phi \in H^{\infty}(\R^2;T_{\m_0} \St)$. Taking into account \eqref{eq:non_deg} and \eqref{eq:add_kernel}, claim (ii) of Proposition \ref{prop:stability} 
follows if $\mathcal{H}(\bs \phi)=0$ and $\bs \phi(0)=0$ implies $\bs \phi \equiv 0$ in $\R^2$. By virtue of the truncation lemma, this will also follow from the estimates below.


\subsection{Moving frames}
We project the tangent field $\bs\phi$ onto an appropriate orthogonal frame in $T_{\m_0} \St$. For the smooth axisymmetric critical point of the form \eqref{eq:m0}
we choose smooth tangent vector fields
\begin{align*}
\bs X = (\cos \psi , \sin \psi , 0) \quad \text{and} \quad
\bs Y = (-\sin \psi \cos \theta(r), \cos \psi \cos\theta(r), - \sin \theta(r))
\end{align*}
on $\R^2 \setminus \{0\}$, satisfying
\[
\m_0 = \bs X \times \bs Y, \quad |\bs X|=|\bs Y|=1, \quad \bs X \cdot \bs Y=0.
\]
Hence $(\bs X,\bs Y)$ forms a smooth orthogonal frame for $T_{\m_0} \St$ on $\R^2 \setminus \{0\}$. Tangent fields $\bs\phi \in C^\infty_c(\R^2 \setminus \{0\};T_{\m_0} \St)$ may be written as 
\[
\bs\phi = u_1\bs X + u_2 \bs Y
\] 
with uniquely determined coefficient functions
$u=(u_1, u_2)  \in  C_c^\infty(\R^2 \setminus \{0\};\R^2)$. Using the relations
\begin{align*}
\bs X \cdot \del_r \bs Y = \del_r \bs X \cdot \bs Y=0 \quad \text{and}\quad \del_{\psi} \bs X \cdot \bs Y = - \bs X \cdot \del_{\psi} \bs Y = \cos \theta,
\end{align*}
by a straightforward calculation we obtain
\[
|\nabla \bs \phi|^2 = |\nabla u|^2 + 2\frac{\cos \theta}{r^2}(u \times \del_{\psi} u) + \frac{1}{r^2} u_1^2 + \left( (\theta')^2 + \frac{\cos^2 \theta}{r^2} \right) u_2^2,
\]
\[
\bs \phi \cdot (\nabla \times \bs \phi) = -\frac{1}{r}\sin \theta (u \times \del_{\psi} u) + \left( \theta' - \frac{1}{r}\sin \theta \cos \theta \right) u_2^2
\]
and
\[
\Lambda(\m_0)
=(\theta')^2 + \frac{1}{r^2} \sin^2 \theta + 2\theta' + \frac{2}{r} \sin \theta \cos \theta +h(1- \cos \theta).
\]
Thus in these coordinate the Hessian reads
\[
\mathcal{H}(\bs\phi) = \int_0^{2\pi} \int_0^{\infty} \left\{ |\nabla u|^2 
 + 2\left( \frac{\cos \theta}{r^2} - \frac{\sin \theta}{r} \right) (u \times \partial_{\psi}u) 
 + f(r,\theta) u_1^2 + g(r,\theta) u_2^2 \right\} r \dif r \dif\psi 
\]
where
\[
f(r,\theta) = \frac{1}{r^2} \cos^2 \theta -  (\theta')^2 - 2\theta' - \frac{2}{r} \sin \theta \cos \theta  + h\cos \theta,
\]
and
\[
g(r,\theta) = \frac{1}{r^2}(\cos^2 \theta - \sin^2 \theta) - \frac{4}{r} \sin \theta \cos \theta + h \cos \theta.
\]

\subsection{Fourier expansion} We expand $u_i$ in a Fourier series
\[
u_i(r,\psi) = \alpha_i^{(0)}(r) + \sum_{k =1}^{\infty} (\alpha_i^{(k)}(r) \cos(k\psi) + \beta_i^{(k)}(r) \sin(k \psi) ),
\]
where $\alpha_i^{(k)}, \beta_i^{(k)}  \in C_c^\infty(0, \infty)$. Due to the $L^2$-orthogonality 
\begin{align*}
& \int_0^{2\pi} u_i^2 \dif \psi = \pi \left( 2(\alpha_i^{(0)})^2+ \sum_{k =1}^{\infty}\left((\alpha_i^{(k)})^2 + (\beta_i^{(k)})^2 \right)\right), \\
& \int_0^{2\pi} |\del_r u_i|^2 \dif \psi = \pi \left( 2 \left((\alpha_i^{(0)})'\right)^2 +\sum_{k =1}^{\infty}\left(\left((\alpha_i^{(k)})'\right)^2 + \left((\beta_i^{(k)})'\right)^2 \right)\right), \\
& \int_0^{2\pi} |\del_{\psi} u_i|^2 \dif \psi = \pi \sum_{k =1}^{\infty}k^2\left((\alpha_i^{(k)})^2 + (\beta_i^{(k)})^2 \right), \\
\text{and} \quad &
\int_0^{2\pi} (u \times \del_{\psi} u) \dif \psi = \pi \sum_{k =1}^{\infty} 2k\left(\alpha_1^{(k)} \beta_2^{(k)} - \alpha_2^{(k)} \beta_1^{(k)} \right).
\end{align*}
Moreover we obtain
\begin{equation}\label{eq:H1}
 \mathcal{H}(\bs\phi)  =
 2\pi \, \mathcal H_0(\alpha^{(0)}_1, \alpha^{(0)}_2)+ \pi \sum_{k=1}^{\infty} \left( \mathcal H_k(\alpha_1^{(k)}, \beta_2^{(k)}) + \mathcal H_k(\beta_1^{(k)}, -\alpha_2^{(k)}) \right)
\end{equation}
where
\begin{align}\label{eq:H_k}
\mathcal  H_k(\alpha, \beta) = \int_0^{\infty} \bigg\{ & 
 |\alpha'|^2  + |\beta'|^2 
 +\left( \frac{k^2}{r^2} + f(r,\theta) \right) \alpha^2 
 +\left( \frac{k^2}{r^2} + g(r,\theta) \right) \beta^2  \\ \nonumber
 &  + 4k \left( \frac{\cos \theta}{r^2} - \frac{\sin \theta}{r} \right) \alpha \beta 
\bigg\} r \dif r.
\end{align}

The following Lemma provides the reduction to the first two modes.
\begin{lemma}\label{la:reduction} For $k \ge 1$ we have
$
\mathcal H_{k+1}(\alpha,\beta) \geq \mathcal H_k(\alpha,\beta)
$
for all $\alpha,\beta \in C^\infty_c(0,\infty)$. More precisely, the inequality holds pointwisely for the
corresponding integrands.
\end{lemma}

\begin{proof}
For any $\alpha,\beta \in C^\infty_c(0,\infty)$ we have
\begin{align*}
 \mathcal H_{k+1}(\alpha,\beta) - \mathcal  H_k(\alpha,\beta)
= \int_0^{\infty} \left\{ \frac{2k+1}{r^2} (\alpha^2 + \beta^2)+\frac{4}{r^2}(\cos \theta - r\sin \theta) \alpha \beta  \right\} r \dif r 
\end{align*}
For $k \geq 1$, the pointwise estimate (\ref{eq:theta_cos}) yields the following pointwise estimate
\begin{align*}
& \frac{1}{r^2} \left((2k+1) (\alpha^2 + \beta^2)+4(\cos \theta - r\sin \theta)\alpha \beta\right) \\
\geq  & \frac{1}{r^2} \left( (2k+1) (\alpha^2 + \beta^2)-6|\alpha \beta| \right)
\geq \frac{3}{r^2}(|\alpha|-|\beta|)^2 \geq 0
\end{align*}
proving the claim.
\end{proof}

\subsection{Non-negativity of the reduced functional}

It remains to consider the first two modes for $k=0,1$. To simplify this task we need the following decomposition, 
which is proven and called as ``Hardy-type decomposition" in \cite{ignat2015stability}. We modify the conditions to suit our needs.
\begin{lemma}\label{la:Hardy}
Let $A:(0,\infty) \to \R$ be a nonnegative $C^1$ function and $V \in L_{\rm loc}^1((0,\infty),\R)$. Define the operator
\[
L := -\frac{\dif}{\dif r} (A(r) \frac{\dif}{\dif r} ) + V(r)
\]
and consider a smooth function $\psi: (0,\infty) \to \R$ satisfying $\psi>0$ in $(0,\infty)$. For every $f \in C_c^{\infty}(0,\infty)$, writing
\[
g:=\frac{f}{\psi} \in C_c^{\infty}(0,\infty),
\]
following decomposition holds true:
\begin{equation}\label{eq:Hardy}
\int_0^{\infty} Lf \cdot f \dif r = \int_0^{\infty} \psi^2 A(r)(g')^2 \, \dif r + \int_0^{\infty} g^2 L\psi \cdot \psi \, \dif r.
\end{equation}
\end{lemma}
It is customary to represent admissible $\alpha, \beta$ as
\begin{equation} \label{eq:xi_eta}
\alpha=\frac{1}{r}\sin \theta\xi \quad \text{and} \quad \beta=-\theta' \eta,
\end{equation}
for some uniquely determined $\xi,\eta \in C_c^{\infty}(0,\infty)$.

\medskip

With $\alpha=\frac{1}{r}\sin \theta\xi$ and $\beta=-\theta' \eta$, where $\xi,\eta \in C_c^{\infty}(0,\infty)$, it follows from Lemma \ref{la:Hardy} with $A=r$, $V=\frac{k^2}{r}+f(r,\theta)r$, $\psi=\frac{1}{r}\sin \theta$ that
\begin{align*}
& \int_0^{\infty} \left\{ r|\alpha'|^2  + \left( \frac{k^2}{r} + rf(r,\theta) \right) \alpha^2  \right\} \dif r \\
= & \int_0^{\infty} \left\{ \frac{\sin \theta^2}{r}    (\xi')^2   +  \frac{k^2-1}{r}  \left(\frac{\sin \theta}{r}  \right)^2 \, \xi^2
+ 2\frac{\sin \theta}{r}  \, \theta' \left( \frac{\cos \theta}{r} - \sin \theta \right) \xi^2\right\}  \dif r
\end{align*}
and respectively with $A=r$, $V=\frac{k^2}{r}+g(r,\theta)r$ and $\psi=-\theta'$ that
\begin{align*}
& \int_0^{\infty} \left\{ r|\beta'|^2 +\left( \frac{k^2}{r} + rg(r,\theta) \right) \beta^2 \right\}  \dif r\\
=& \int_0^{\infty} \left\{ r(\theta')^2  (\eta')^2
+ \frac{k^2-1}{r} (\theta')^2 \eta^2 
+ 2\frac{\sin \theta}{r}  \, \theta' \left( \frac{\cos \theta}{r} - \sin \theta  \right)  \eta^2 \right\} \dif r.
\end{align*}
Altogether we have
\begin{align*}
& \mathcal H_k (\alpha,\beta) = \tilde{\mathcal H}_k(\xi,\eta) \\
 =& \int_0^{\infty} \Bigg\{ \frac{\sin^2 \theta }{r} (\xi')^2 
+ r(\theta')^2 (\eta')^2 
 + \frac{k^2-1}{r} \left(\frac{\sin \theta}{r}\right)^2  \, \xi^2 + \frac{k^2-1}{r} (\theta')^2 \eta^2 \\
&\qquad\qquad + 2\frac{\sin \theta}{r}  \, \theta' \left( \frac{\cos \theta}{r} - \sin \theta \right) (\xi^2 - 2k\xi \eta + \eta^2)
\Bigg\} \dif r.
\end{align*}

\begin{lemma}\label{la:PositiveH0}
$\tilde{\mathcal H}_0(\xi,\eta) > 0$ for all $ \xi,\eta \in C^{\infty}_c(0,\infty) \setminus \{0\}$. 
\end{lemma}
\begin{proof}
For 
\begin{align*}
 \tilde{\mathcal H}_0(\xi,\eta) 
 =& \int_0^{\infty} \Bigg\{ \frac{\sin^2 \theta }{r}  (\xi')^2 
+ r(\theta')^2 (\eta')^2 
 - \frac{1}{r} \left(\frac{\sin \theta}{r}\right)^2  \, \xi^2 -\frac{1}{r} (\theta')^2 \eta^2 \\
&\qquad\qquad + 2\frac{\sin \theta}{r}  \, \theta' \left( \frac{\cos \theta}{r} - \sin \theta \right) (\xi^2 + \eta^2)
\Bigg\} \dif r
\end{align*}
we apply the decomposition (\ref{eq:Hardy}) for $A=\frac{\sin^2 \theta}{r}$, $V=0$ and $\psi=r$ and get
\begin{align*}
 & \int_0^{\infty} \frac{\sin^2 \theta}{r} (\xi')^2 \dif r 
= \int_0^{\infty} -\left(\frac{\sin^2 \theta}{r} \xi' \right)' \xi \dif r \\
= & \int_0^{\infty} r^2 \frac{\sin^2 \theta}{r} \left( \left(\frac{\xi}{r} \right)' \right)^2 + \left(\frac{\xi}{r}\right)^2 \left(\frac{\sin^2 \theta}{r} - 2\sin \theta \cos \theta \theta' \right) \dif r,
\end{align*}
hence
\begin{equation}\label{eq:H0_xi}
\int_0^{\infty} \left\{ \frac{\sin^2 \theta}{r} \left( (\xi')^2 - \frac{1}{r^2} \xi^2 \right) 
+ \frac{2}{r^2} \sin \theta \cos \theta \, \theta' \xi^2 \right\} \dif r \geq 0.
\end{equation}
Analogously it follows from (\ref{eq:Hardy}) with $A=r(\theta')^2$, $V=0$ and $\psi=r$ and (\ref{eq:thetaODE}) for $\theta$ that
\begin{align*}
 &\int_0^{\infty} r (\theta')^2 (\eta')^2 \dif r 
=  \int_0^{\infty}  - \left(r \left(\theta'\right)^2 \eta' \right)' \eta  \dif r \\
=& \int_0^{\infty} r^2 r (\theta')^2 \left( (\frac{\eta}{r})' \right)^2 \dif r
+ \int_0^{\infty} \left( \frac{\eta}{r} \right)^2 \left(-r(\theta')^2 - 2r^2 \theta' \theta'' \right) \dif r\\
= &\int_0^{\infty} r^3 (\theta')^2 \left( (\frac{\eta}{r})' \right)^2
+ \left( \frac{\eta}{r} \right)^2 \left( r(\theta')^2 
- 2\sin \theta \cos \theta \, \theta' + 4r \sin^2 \theta \, \theta' -2hr^2 \sin \theta \, \theta' \right) \dif r
\end{align*}
and thus by the estimate (\ref{eq:theta_h}) for $\theta$ 
\begin{equation}\label{eq:H0_eta}
\begin{aligned}
& \int_0^{\infty} \left\{ r (\theta')^2 (\eta')^2 - \frac{1}{r} (\theta')^2 \eta^2 + \frac{2}{r^2} \sin \theta \cos \theta \, \theta' \eta^2 + \frac{\sin^2 \theta}{r}  (- \theta')  \eta^2 \right\} \dif r  \\
= & \int_0^{\infty}  r^3 (\theta')^2 \left(\left(\frac{\eta}{r}\right)'\right)^2 +2 (-\theta') \sin  \theta \left( h-\frac{3}{2r} \sin \theta \right) \eta^2 \dif r \geq 0.
\end{aligned}
\end{equation}
Adding \eqref{eq:H0_xi} and \eqref{eq:H0_eta} we get
\[
\tilde{\mathcal H}_0(\xi,\eta) \geq \int_0^{\infty}  2\frac{\sin^2 \theta}{r}  (- \theta') (\xi^2  + \frac{1}{2}\eta^2) \dif r \geq 0,
\]
and the claim follows.
\end{proof}

\begin{lemma}\label{la:PositiveH1}
$\tilde{\mathcal H}_1(\xi,\eta) > 0$ for all $\xi,\eta \in C^{\infty}_c(0,\infty) \setminus \{0\}$.
\end{lemma}
\begin{proof}
By partial integration we have
\begin{align*}
\tilde{\mathcal H}_1(\xi,\eta) =&\int_0^{\infty} \frac{\sin^2 \theta}{r}(\xi')^2+r(\theta')^2 (\eta')^2 + 2\frac{\sin \theta}{r} \theta' \left(\frac{\cos \theta}{r}-\sin \theta \right) (\xi-\eta)^2  \dif r \\
=& \int_0^{\infty} \frac{\sin^2 \theta}{r}(\xi')^2+r(\theta')^2 (\eta')^2 + \frac{2\sin^2 \theta}{r^3} (\xi-\eta)^2 + \frac{2\sin^2 \theta}{r}(-\theta')(\xi-\eta)^2 \\
& \qquad\quad  -\frac{2 \sin^2 \theta}{r^2}(\xi-\eta)\xi' + \frac{2 \sin^2 \theta}{r^2} (\xi-\eta) \eta' \dif r\\
=& \int_0^{\infty} \left( \frac{\sin \theta}{r^{1/2}} \xi' - \frac{\sin \theta}{r^{3/2}}(\xi-\eta) \right)^2 
+ \left(r(\theta')^2 - \frac{\sin^2 \theta}{r(1+2r^2(-\theta'))} \right)(\eta')^2
\\
& \quad + \left( \frac{\sin \theta}{r^{1/2}}(1+2r^2(-\theta'))^{-\frac{1}{2}} \eta'+  \frac{\sin \theta}{r^{3/2}}(1+2r^2(-\theta'))^{\frac{1}{2}}(\xi-\eta) \right)^2 \dif r.
\end{align*}
For the second term it follows from \eqref{eq:theta_sin} and the monotonicity of $\theta$
\[
r(\theta')^2 - \frac{\sin^2 \theta}{r(1+2r^2(-\theta'))}
= \frac{1}{r(1+2r^2(-\theta'))} \left( r^2(\theta')^2 - \sin^2 \theta  - 2r^3 (\theta')^3 \right)>0
\]
for every $r \in (0,\infty)$. Hence $\tilde{\mathcal H}_1(\xi,\eta) \ge 0$ with equality only if $\xi=\eta \equiv 0$.
\end{proof}

To conclude claim (ii) of Proposition \ref{prop:stability}, we consider for $\bs \phi \in H^{\infty}(\R^2;T_{\m_0}\St)$ with $\bs \phi(0)=0$ and
$\mathcal{H}(\bs \phi)=0$ a family of approximating truncations $\bs \phi_\eps$ such that $\mathcal{H}(\bs \phi_\eps) \to 0$ as $\eps  \to 0$, according to Lemma \ref{lemma:approximation}. Regarding the zeroth mode, the estimate in the proof of Lemma \ref{la:PositiveH0} implies that the corresponding functions $\xi_\eps$ and $\eta_\eps$ converge
to zero in measure. Regarding the first and higher modes, the estimate in the proof of Lemma \ref{la:PositiveH1} implies that the corresponding functions $\xi_\eps$ and $\eta_\eps$ 
converge in measure to a constant. But according to \eqref{eq:xi_eta} and 

\[
\lim_{r \to 0}\frac{\sin \theta(r)}{r}= - \lim_{r \to 0} \theta'(r) =\frac{h}{2}, 
\]
according to Lemma \ref{la:thetaAsymtot}, this constant must be zero. 

\section{Spectral gap and Fredholm property}
\label{sec:fred}

In this section we complete the proof of Theorem \ref{thm:stability}.

\begin{lemma} \label{lemma:wlsc} \mbox{}
\begin{itemize}
\item[(i)] For $h>1$ the form $\mathcal H_{\infty}$ is subcritical in the sense that for some positive $\gamma$ and $\nu$
\begin{equation*}
\mathcal H_{\infty}(\bs \phi)-\gamma \| \bs \phi\|_{L^2}^2 \ge \nu \| \bs \phi\|_{H^1}^2.
\end{equation*}
\item[(ii)] 
The form $\Delta \mathcal{H}$ is compact in the sense that for $\bs \phi_k \rightharpoonup \bs \phi$  weakly in $H^1(\R^2; \R^3)$
\begin{equation*}
\lim_{k \to \infty} \Delta \mathcal{H}(\bs \phi_k) =\Delta \mathcal{H}(\bs \phi). 
\end{equation*}
\item[(iii)] The functional
$
\mathcal{H}(\bs \phi) - \gamma \|\bs \phi\|_{L^2}^2 
$
is weakly lower semincontinuous in $H^1(\R^2; \R^3)$.
\end{itemize}
\end{lemma}

\begin{proof}
Applying Young's inequality to $\bs \phi \cdot (\nabla \times \bs \phi)$ and using the orthogonality relation $\|\nabla \times \bs \phi\|_{L^2} =
\|\nabla  \bs \phi\|_{L^2}$ we obtain 
\begin{equation}\label{eq:prel_lower_infty}
\mathcal{H}_\infty(\bs \phi) \ge  \frac{h-1}{h}  \max\{ \|\nabla \bs \phi\|_{L^2}^2, h \|\bs \phi\|_{L^2}^2 \} \ge  \frac{h-1}{h+1} \|\bs \phi\|^2_{H^1} 
\end{equation}
Claim (i) follows with $\gamma=(h+1)/2$ and $\nu=(h'-1)/(h'+1)$ where $h'=(h-1)/2$.

\medskip

By virtue of the Rellich--Kondrachov theorem $\int_{B_R} \Lambda(\m_0) |\bs \phi_k|^2 \dif x$ is compact for every finite $R$.
Since $\Lambda(\m_0) \in L^2(\R^2)$ by Proposition  \ref{prop:H^2}  and $\|\bs \psi\|_{L^4} \lesssim \|\bs \psi\|_{H^1}$, we have
\[
 \left| \mathcal H_{\infty}(\bs \psi, \bs \psi) - \int_{B_R} \Lambda(\m_0) |\bs \psi|^2 \dif x \right| \lesssim \|\bs \psi\|_{H^1}^2 \|\Lambda(\m_0)\|_{L^2(B_R^c)} \to 0
\]
as $R \to \infty$. Since by assumption $\sup_{k \in \N} \|\bs \phi_k\|_{H^1}<\infty$, claim (ii) and (iii) follow. 
\end{proof}

We also need the following regularity result for the weak equation $\mathcal{J}\bs \phi =\bs f$.

\begin{lemma}\label{lemma:H^2}
Suppose $\bs f \in L^2(\R^2;T_{\m_0}\St)$. If $\bs \phi \in H^1(\R^2;T_{\m_0}\St)$ satisfies
\[
\mathcal{H}(\bs \phi, \bs \psi)= \langle \bs f, \bs \psi \rangle_{L^2} \quad \text{for all} \quad \bs \psi \in H^1(\R^2; T_{\m_0}\St),
\]
then $\bs \phi \in H^2(\R^2;\R^3)$.
\end{lemma}

\begin{proof} Clearly $\mathcal{H}(\bs \phi, P_{\m_0} \bs \eta)= \langle \bs f, \bs \eta \rangle_{L^2}$ for all $\bs \eta \in C_0^\infty(\R^2;\R^3)$ and
\[
\mathcal{H}(\bs \phi, P_{\m_0} \bs \eta) = \langle \nabla \bs \phi, \nabla(P_{\m_0} \bs \eta) \rangle_{L^2}-\langle 
\bs g_0, \bs \eta \rangle_{L^2}
\]
for some $\bs g_0 \in L^2(\R^2;\R^3)$ bounded in terms of $\m$ and $\bs \phi$. Moreover, for smooth
$\bs \phi$ and $\m$ we have 
$\langle \nabla \bs \phi, \nabla(P_{\m} \bs \eta) \rangle_{L^2}= - \langle P_{\m}  \Delta \bs \phi, \bs \eta \rangle_{L^2}$ and
\[
P_{\m} \Delta \bs \phi = \Delta \bs \phi - \Delta  (\m \cdot \bs \phi) \m + \nabla \cdot( \nabla \m \cdot \bs \phi) \m +(\nabla \m : \nabla \bs \phi) \m
\]
hence by approximation and using $\bs \phi \cdot \m_0=0$
\[
 \langle \nabla \bs \phi, \nabla(P_{\m_0} \bs \eta) \rangle_{L^2} = \langle  \nabla \bs \phi, \nabla \bs \eta \rangle_{L^2} + \langle \bs g_1, \bs \eta \rangle_{L^2},
\]
where $\bs g_1 \in L^2(\R^2;\R^3)$ by Proposition \ref{prop:H^2}. Hence $-\Delta \bs \phi = 
\bs f +\bs g_0+\bs g_1$, and
the claim follows from standard $L^2$ theory for the Poisson equation.
\end{proof}

Recall that $\ker J= \{c \cdot \nabla \m_0: c \in \R^2\}$. In this context we shall use the notation 
\[\langle \bs \phi, \nabla \m_0\rangle_{L^2} 
= \sum_{j=1,2} \langle \bs \phi, \del_j \m_0\rangle_{L^2} \ein_j  \in \R^2.
\]

\begin{proposition}\label{prop:gap}
For $h\gg 1$ there exists a $\lambda>0$ increasing in $h$ such that
\[
\mathcal{H}( \bs \phi) \geq \lambda \|\bs \phi\|_{H^1}^2
\quad \text{for all} \quad \bs \phi  \in H^1(\R^2;T_{\m_0} \St) \text{  with  }  \langle \bs \phi, \nabla \m_0\rangle_{L^2}=0.\]
$\mathcal{J}: H^2(\R^2;T_{\m_0}\St) \to L^2(\R^2;T_{\m_0}\St) $ is a Fredholm operator
of index $0$ with
\[
\mathrm{ran} \; \mathcal{J} = \{\bs f \in L^2(\R^2;T_{\m_0} \St):
 \langle \bs f,  \nabla \m_0\rangle_{L^2}=0 \}.
\]
\end{proposition}

\begin{proof}
For $k=0,1,2$ we define the Hilbert spaces 
\[
\mathbb{H}^k=\{ \bs \phi \in H^k(\R^2;T_{\m_0} \St):  \langle \bs \phi,  \nabla \m_0\rangle_{L^2}=0\}.
\]
We let $I_0:=\inf\{ \mathcal{H}(\bs \phi): \, \bs \phi \in \mathbb{H}^1, \, \|\bs \phi\|_{L^2} = 1\}\ge 0$
and suppose $I_0 = 0$. Consider
\[
I_{\gamma} := \inf\{ \mathcal H_{\gamma}(\bs \phi): \, \bs \phi \in \mathbb{H}^1, \, \|\bs \phi\|_{L^2}=1\} 
\]
where
\[
\mathcal H_{\gamma}(\bs \phi)= \mathcal{H}(\bs \phi)-\gamma \|\bs \phi\|_{L^2}^2,
\]
which is an $H^1$ norm and therefore weakly lower semicontinuous for some positive $\gamma$ by virtue of Lemma \ref{lemma:wlsc}.
If $I_\gamma=0$, then $\mathcal{H}(\bs \phi) \ge \gamma \|\bs \phi\|_{L^2}^2$ which implies $I_0>0$. Hence $I_\gamma<0$,
and we claim the infimum is attained. From Lemma \ref{lemma:wlsc} (i) we also obtain
\begin{equation}\label{eq:prel_lower}
\mathcal{H}(\bs \phi) \ge \mathcal H_\gamma (\bs \phi)  \ge \nu \|\bs \phi\|_{H^1}^2 - \mu \|\bs \phi\|_{L^2}^2
\end{equation}
with $\nu >0$ and $\mu= \|\Lambda(\m_0)\|_{L^\infty}^2$. Since $\mathcal H_{\gamma}(\bs \phi) \geq \nu \|\bs \phi \|_{H^1}^2 - \mu $ for $\|\bs \phi\|_{L^2}=1$, there exists for $I_\gamma$ a minimizing sequence
\[
\bs \phi_k \rightharpoonup \bs \phi_\ast \quad\text{weakly in } H^1(\R^2;T_{\m_0} \St).
\]
Clearly $\bs \phi_\ast \in \mathbb{H}^1$ and $\|\bs \phi_\ast\|_{L^2}^2 \leq \liminf_{k \to \infty} \|\bs \phi_k\|_{L^2}^2 =1$. Since $I_{\gamma}<0$ and $\mathcal H_\gamma(0)=0$, 
it follows from weak lower semicontinuity that $\bs \phi_\ast \not=0$. Therefore
\[
\|\bs \phi_\ast\|_{L^2}^2 \, I_{\gamma} \leq  \mathcal H_{\gamma} (\bs \phi_\ast)  \leq \liminf_{k \to \infty} \mathcal H_{\gamma} (\bs \phi_k)=I_{\gamma}
\]
which implies that $\|\bs \phi_\ast \|_{L^2}^2 = 1$ and $I_{\gamma}=H_\gamma(\bs \phi_\ast)$. Hence
$\mathcal{H}(\bs \phi) \geq \mathcal{H}(\bs \phi_\ast)$ for all $\|\bs \phi\|_{L^2}=1$, and therefore
\[
\mathcal{H}(\bs \phi)  \geq \mathcal{H}(\bs \phi_\ast) \|\bs \phi \|_{L^2}^2 \quad \text{for all} \quad \bs \phi \in \mathbb{H}^1 \setminus \{0\}, 
\]
a contradiction, since $\mathcal{H}(\bs \phi_\ast)>0$.  Therefore $I_0>0$, and together with the lower bound \eqref{eq:prel_lower}, the claim follows with $\lambda=\nu/(1+\mu/I_0)$. \\
It now follows from the Riesz representation theorem that for $f \in \mathbb{H}^0$ there exists a unique $\bs \phi \in \mathbb{H}^1$ with
$\mathcal{H}( \bs \phi, \bs \psi) = \langle f, \bs \psi \rangle_{L^2}$ for all $\bs \psi \in  \mathbb{H}^1$, while Lemma \ref{lemma:H^2} implies that $\bs \phi \in \mathbb{H}^2$ and hence $\mathcal{J} \bs \phi = \bs f$. Since $\mathcal{J}:\mathbb{H}^2 \to \mathbb{H}^0$ is bounded, it is an isomorphism, and the claim follows.
\end{proof}

\section{Local minimality}\label{sec:min}
In this section we prove Theorem \ref{thm:minimality}.
The key observation is that the energy difference can be expressed in terms of the extended Hessian. This is in the spirit of \cite{mcintosh_simon1990, hardt1992harmonic}, but with stronger conclusions.

\begin{lemma} \label{lemma:quadratic}
For $\m \in H^1_{\ein_3}(\R^2;\St)$ we have
$
E(\m) - E(\m_0) = \frac{1}{2} \mathcal{H}(\m- \m_0).
$
\end{lemma}

\begin{proof}
Define $\bs \xi = \m-\m_0 \in H^1(\R^2;\R^3)$. Then
\[
 E(\bs m )-E(\bs m_0)  =\frac{1}{2} \mathcal H_{\infty}(\bs \xi) + \mathcal H_{\infty}(\bs \xi, \bs m_0 -\ein_3)
\]
and using that $\bs \tau(\m_0)=0$
\[
\mathcal H_{\infty}(\bs \xi, \bs m_0 -\ein_3)= \int_{\R^2}\Lambda(\m_0) \m_0 \cdot  \bs\xi \dif x.
\]
It follows that 
\begin{eqnarray*}
 E(\bs m )-E(\bs m_0)  &=& \frac{1}{2} \mathcal H_{\infty}(\bs \xi) -\Delta \mathcal{\mathcal H}(\bs \xi) +\int_{\R^2}\Lambda(\m_0) \m \cdot  \bs\xi \dif x \\
 &=& \frac{1}{2} \mathcal{\mathcal H}(\bs \xi)+\frac{1}{2} \int_{\R^2}\Lambda(\m_0)( \m +\m_0) \cdot ( \m -\m_0) \dif x
\end{eqnarray*}
where the integrand of the last term vanishes identically.
\end{proof}

The requisite orthogonality can be established by an appropriate translation.

\begin{lemma} \label{lemma:trans}
There exist $\eps>0$ and $c >0$ such that for $\m \in H^1(\R^2;\St)$
with $\|\m-\m_0\|_{H^1} <\eps$ there exists a unique $x \in \R^2$ such that
\[
\langle \m(\cdot -x)-\m_0, \nabla \m_0 \rangle_{L^2} =0
\quad
\text{and}
\quad
 \| \m(\cdot -x)-\m_0\|_{H^1} \le  c \eps.
 \]
\end{lemma}

\begin{proof} Existence and uniqueness follow
 from the implicit function theorem applied to the $C^1$ mapping
$F(x, \bs \xi)=\langle \m(\cdot -x)-\m_0, \nabla \m_0 \rangle_{L^2}$, where $\bs \xi=\m-\m_0$, in the vicinity of $(x, \bs \xi)=(0,0)$ in $H^1(\R^2;\R^3) \times \R^2$. 
The Jacobian $[\del_x F(0,0)]_{jk}=- \langle \del_j \m_0, \del_k \m_0 \rangle_{L^2}$ is non-singular in $\R^{2 \times 2}$ since the topological degree $Q(\m_0)\not=0$. Hence there exists a $C^1$ mapping $\bs \xi \mapsto x(\bs \xi)$ such that $F(x(\bs \xi),\bs \xi )=0$  for $ \|\bs \xi\|_{H^1}< \eps$. Implicit differentiation implies $|x(\bs \xi)| \lesssim \| \bs \xi \|_{H^1}$. Hence the
estimate follows from translation invariance of the $H^1$ norm and the fact that
$\|\m_0(\cdot -x)-\m_0\|_{H^1} \lesssim |x|$ according to Proposition \ref{prop:H^2}.
\end{proof}

\begin{proof}[Proof of Theorem \ref{thm:minimality}] By virtue of Lemma \ref{lemma:trans} we can assume $\langle \bs \xi, \nabla \m_0 \rangle_{L^2} =0$ where $\bs \xi =\m-\m_0$. We decompose $\bs \xi =\bs \xi^\top +\bs  \xi^\perp$  
where $\bs \xi^\top=P_{\m_0} \bs \xi$, hence 
\begin{equation} \label{eq:perp}
\bs  \xi^\perp=(\bs \xi \cdot \m_0) \m_0 = - \frac{1}{2}|\bs \xi|^2 \m_0,
\end{equation}
and in particular
\begin{equation} \label{eq:perp_norm}
\| \bs  \xi^\perp \|_{L^2} = \frac{1}{4} \| \bs \xi\|_{L^4}^2 \le c \| \bs  \xi \|_{H^1}^2. 
\end{equation}
It follows from
Lemma \ref{lemma:quadratic} and Proposition \ref{prop:gap}
\[
E(\m) - E(\m_0) \ge\frac{1}{2} \left(\lambda \| \bs \xi^\top \|_{H^1}^2+ \mathcal{H}( \bs  \xi^\perp) 
+ 2\mathcal{H}(\bs \xi^\top, \bs  \xi^\perp) \right). 
\]
By \eqref{eq:perp} we observe that the helicity terms in $\mathcal{H}_\infty( \bs  \xi^\perp)$ and $\Delta\mathcal{H}( \bs  \xi^\perp)$ cancel, so that taking into account \eqref{eq:perp_norm}
\[
\mathcal{H}( \bs  \xi^\perp) =  \int_{\R^2}  | \nabla\bs  \xi^\perp|^2_{L^2} + \frac{h \, (\ein_3 \cdot \m_0) - |\nabla \m_0|^2}{4} |\bs \xi|^4  \dif x \ge  \| \bs  \xi^\perp\|^2_{H^1} - C  \| \bs  \xi\|^4_{H^1},
\]
where $C$ denotes a positive constant that only depends on $\m_0$.
Regarding $\mathcal{H}(\bs \xi^\top, \bs  \xi^\perp)$, we estimate the leading order term 
\begin{eqnarray*}
\left| \int_{\R^2} \nabla \bs \xi^\top : \nabla \bs  \xi^\perp \dif x \right| &\le&  \frac{1}{4}\int_{\R^2} 
|\nabla (P_{\m_0} \bs \xi) : \nabla( |\bs  \xi|^2 \m_0)| \dif x\\
&\le &  \frac{1}{4}\int_{\R^2}  |\bs  \xi|^2 |\nabla \bs \xi| |\nabla \m_0| + |\nabla (|\bs \xi|^2 \m_0)| |\nabla (P_{\m_0})\bs \xi|  \dif x \\
&\le& C \left( \|\nabla \bs \xi\|_{L^2} \|\bs \xi\|_{L^4}^2+\|\bs \xi\|_{L^4}^3 \right) \le C\|\bs \xi\|^3_{H^1},
\end{eqnarray*}
taking into account the orthogonality of $P_{\m_0}(\nabla \bs \xi)$ and $\m_0$.
Improved estimates are valid for the remaining terms, hence $|\mathcal{H}(\bs \xi^\top, \bs  \xi^\perp)| \le  C\|\bs \xi\|^3_{H^1}$. It follows that for $\| \bs \xi \|_{H^1}< \eps<1$
\[
E(\m) - E(\m_0) \ge \frac 1 2 \left( \min\{\lambda, 1\} -C \eps \right) \| \bs \xi \|_{H^1}^2 
\]
which implies the claim for $\eps$ sufficiently small.
\end{proof}

\section{Solitonic motion} \label{sec:trav}
In this section we prove Theorem \ref{thm:wave}. Passing to a moving frame $\m=\m(x-ct)$ for some unknown $c \in \R^2$, the Landau-Lifshitz-Gilbert equation (\ref{eq:LLG}) with spin velocity $v \in \R^2$ becomes a stationary  operator equation
\[
\bs F(\m,c,v) =0
\]
where 
\[
\bs  F(\m,c,v) = - \bs \tau(\m) + \m \times \left[ (v-c) \cdot \nabla \m \right] + (\beta v- \alpha c) \cdot \nabla \m.
\]

It is convenient to project $\bs F(\m,c,v) \in T_{\m} \St$ onto the tangent space $T_{\m_0} \St$. We need the following variant of a lemma from \cite{hardt1992harmonic}:
\begin{lemma}\label{la:Fv_equiv}
For $\bs  F =\bs F(\m,c,v)$ and $\|\m-\m_0\|_{H^2}$ sufficiently small we have
\[
\bs F =0 \iff P_{\m_0} \bs F=0.
\]
\end{lemma}

\begin{proof} Necessity of $P_{\m_0} \bs F=0$ is obvious. Now suppose $P_{\m_0} \bs  F=0$. Then it follows from the projection 
property of $P_{\m_0}$ that
\[
|\bs  F |^2 = \langle \bs  F, (P_{\m} - P_{\m_0}) \bs F \rangle  \le \|P_{\m} - P_{\m_0}\|_{L^\infty} |\bs F |^2,
\]
hence $\bs  F=0$ provided $\|P_{\m} - P_{\m_0}\|_{L^\infty}<1$, which can also be estimated in terms of $\|\m-\m_0\|_{\sup} \lesssim \|\m-\m_0\|_{H^2}$ by Sobolev embedding.
\end{proof}

Assuming $h \gg 1$, we aim to construct $(\m,c)$ for small $v$ in the vicinity of the stationary solution $\m_0$ and $c=0$. To this end, we introduce the operator
\[
\mathcal{F}: H^2_{\ein_3}(\R^2;\St) \times \R^2  \times \R^2 \to L^2(\R^2;T_{\m_0}\St) \times  \St\]
between Hilbert manifolds defined by
\[
\mathcal{F} (\m,c,v) =\begin{pmatrix} P_{\m_0}\bs  F(\m,c,v)  \\ \m(0) \end{pmatrix}.
\]
By means of the implicit function theorem on Hilbert manifolds \cite{abraham2012manifolds}, we solve $\mathcal{F}(\m,c,v)=(0, -\ein_3)$ for small $v$. Since $\mathcal{F} (\m_0,0,0)=(0, -\ein_3)$, we only have to show invertibility of the differential 
\[
\del_{(\m,c)} \mathcal{F}(\m_0,0,0) : H^2(\R^2;T_{\m_0}\St) \times \R^2 \to L^2(\R^2;T_{\m_0}\St) \times T_{\m_0(0)}\St
\]
given by
\[
\del_{(\m,c)} \mathcal{F} (\m_0,0,0)   = 
\begin{pmatrix}
\mathcal{J} &   \bs S \\ \delta_0 & 0
\end{pmatrix},
\]
where $\bs S$ is the vector field given by
\[
\bs S c = - \left[  \m_0 \times ( c \cdot \nabla) \m_0 + \alpha \, (c \cdot \nabla) \m_0 \right].
\]

\begin{lemma}\label{la:isomorphism}
Suppose $\bs f \in L^2(\R^2;T_{\m_0}\St)$ and $\bs V \in  T_{\m_0(0)}\St$.
Then
\[
\mathcal{J} \bs u + \bs S c = \bs f \quad \text{and} \quad \bs u(0)=\bs V
\]
has a unique solution
\[
\bs u \in H^2(\R^2;T_{\m_0}\St) \quad \text{and}  \quad c \in \R^2.
\]
\end{lemma}

\begin{proof}
Proposition \ref{prop:gap} implies that $c \in \R^2$ is determined by the Fredholm condition 
\[
\bs f - \bs S c \perp  \ker \mathcal{J}=\mathrm{span}\{ \del_1 \m_0, \del_2 \m_0\},
\]
which may be written as a linear $2 \times 2$ system $Ac+b=0$ with
\[
b_j:=\langle \bs f,  \del_j \m_0 \rangle_{L^2}
\quad
\text{and} 
\quad
A_{jk} := 4 \pi \epsilon_{jk} + \alpha  \, \langle \del_j \m_0, \del_k \m_0 \rangle_{L^2}.
\] 
Here we have used that $Q(\m_0)=-1$ while
\[
\langle \del_j \m_0, \m_0 \times \del_k \m_0 \rangle_{L^2}= -4\pi Q(\m_0) \epsilon_{jk}.
\]
It follows from Cauchy-Schwarz that $\det A  > (4\pi)^2$ for all $\alpha>0$, which implies unique solvability. Then the equation for $\bs u$
\[
\mathcal{J} \bs u = \bs f - \bs S c
\]
has a solution $\bs u_0$ which is unique up to an element in $\ker \mathcal{J}$. 
Thanks to \eqref{eq:non_deg}, the second equation
\[
 \bs u(0)= \bs u_0(0)  + \sum_{j=1,2} \lambda_j \del_j \m_0(0) = \bs V
\]
selects a unique solution $\bs u = \bs u_0+ \sum_{j=1,2} \lambda_j \del_j \m_0$.
\end{proof}

We conclude that for $h \gg 1$ there exists $\eps>0$ and a smooth map 
\[
 v \mapsto (\m(v),c(v)) \in H^2_{\ein_3}(\R^2; \St) \times \R^2
\]
 with $\m(0)= \m_0$ and $c(0)=0$ such that $F(\m(v),c(v),v)=0$ for $|v|< \eps$.

\subsection*{Acknowledgements} We thank Radu Ignat for valuable discussions on the subject matter.
This work is partially support by Deutsche Forschungsgemeinschaft (DFG grant no. ME 2273/3-1).

\appendix
\section[Properties of the polar profile]{Properties of the polar profile}
Here we provide the proof of Proposition \ref{thm:theta}.
\begin{lemma}\label{la:thetaMonoton}
If $\theta$ is a solution of (\ref{eq:thetaODE})-(\ref{eq:thetaBoundary}), then
\[
0<\theta(r)<\pi \quad \text{for all } r \in (0,\infty) 
\]
and  $\theta$ is monotonically decreasing on $(0,\infty)$ for $h \gg 1$.
\end{lemma}

\begin{proof}
We make a fine analysis of the ordinary equation in the spirit of the argument for the radial solution of a Ginzburg-Landau-type equation in \cite{hang2001static}. The rescaled function $\phi(t)=\pi-\theta(e^{-t})$ satisfies the second order differential equation
\begin{equation}\label{eq:rescaled}
\phi''(t) = -\sin \phi(t) f(t), \quad \text{where } f(t):=  he^{-2t} - 2e^{-t} \sin \phi - \cos \phi
\end{equation}
with the boundary conditions
\[
\lim_{t \to -\infty} \phi(t) = \pi, \quad
\lim_{t \to +\infty} \phi(t) = 0.
\]
We choose $t_0 = \frac{1}{2}\log \frac{h}{3}$ so that $f(t)>0$ for any $t<t_0$. Moreover, for $h$ large enough we have
\[
|\cos \phi(t)-1| < \varepsilon, \quad |\sin \phi(t)| < \varepsilon \quad \text{and} \quad \phi(t) \in (-\frac{\pi}{4},\frac{\pi}{4}) \quad \text{for all } t<t_0
\]
with some small $\varepsilon>0$, e.g. $\varepsilon=\frac{1}{4}$. We have $\phi(t)<\pi$ for all $t \in (-\infty,+\infty)$. Otherwise there exists a $t^* < t_0$ with $\phi(t^*) \in (\pi,2\pi)$ and $\phi'(t^*)<0$. Then $\phi''(t^*)>0$ and $\phi$ will be oscillating around $2\pi$ as $t \to -\infty$.

We choose $t_1=\frac{1}{2}\log 2h > t_0$ so that $f(t)<0$ for all $t>t_1$. Then either $\phi(t)>0$ for all $t \geq t_1$ or $\phi(t)<0$ for all $t \geq t_1$. If this is not the case, then there exists a $t^*>t_1$ such that $\phi(t^*)=0$, and hence $\phi''(t^*) = 0$. Then $\phi'(t^*)$ must be nonzero according to the uniqueness theorem for ordinary differential equations. If $\phi'(t^*)>0$, then $\phi(t)>0$ and $\phi'(t)>0$ for $t>t^*$ and very close to $t^*$. Hence $f(t)<0$ one can deduce that $\phi$ will be oscillating around $\pi$ as $t \to \infty$, which is a contradiction. Similarly, $\phi'(t^*)<0$ also leads to a contradiction.

Now we assume $\phi(t)>0$ for any $t \geq t_1$, then $\phi'(t)<0$ for $t \geq t_1$. We claim $\phi'(t)<0$ for any $t \in \R$; then it follows that $\phi(t) \in (0,\pi)$.  If $\phi'$ vanishes at some points, set $s_1=\inf \{ t \in \R, \, \phi'(t) = 0 \}$. Since $\phi(t)<\pi$ for any $t\in \R$ and $\phi(-\infty) = \pi$ we can easily deduce that
$-\infty < s_1 < t_1$. 
If $\phi(s_1) \in (2k \pi, (2k+1)\pi)$ for some non-positive integer $k$, then $\phi''(s_1)>0$ and it follows by the choice of $t_1$ that $\sin \phi(s_1) \cos \phi(s_1)  = \sin \phi(s_1) e^{-s_1}(he^{-s_1}-2 \sin \phi(s_1))+\phi''(s_1)>0$ and hence $\phi(s_1) \in (2k \pi, 2k\pi+\frac{\pi}{2})$. Therefore there exists some $s_2 < s_1$ such that $\phi(s_2)=4k\pi + \pi - \phi(s_1)$.  On $[s_2,s_1]$, since $\phi'<0$ and $\sin \phi>0$ we get $\sin \phi \cos \phi - \phi''>0$ and hence $\sin^2 \phi(s_2) - (\phi'(s_2))^2  > \sin^2 \phi(s_1) - (\phi'(s_1))^2$ which yields a contradiction $- (\phi'(s_2))^2>0$.
If $\phi(s_1) \in ((2k-1) \pi, 2k\pi)$ for some non-positive integer $k$, there exist $s_3<s_2<s_1$ such that $\phi(s_2)=2k\pi$ and $\phi(s_3) = 4k\pi-\phi(s_1)$. On $[s_3,s_2]$ we have $\phi''-\sin \phi \cos \phi + he^{-2s_3} \sin \phi \geq \phi''-\sin \phi \cos \phi + he^{-2t} \sin \phi - 2e^{-t} \sin^2 \phi = 0$ and hence $\frac{1}{2}(\phi'(s_3))^2 - \frac{1}{2}\sin^2 \phi(s_3) - he^{-2s_3} \cos \phi(s_3) \leq \frac{1}{2}(\phi'(s_2))^2 -h e^{-2s_3}$. Similarly we get $\frac{1}{2}(\phi'(s_2))^2 -he^{-2s_1} \leq -\frac{1}{2} \sin^2 \phi(s_1) - h e^{-2s_1} \cos \phi(s_1)$. These two inequalities imply $\phi'(s_3)=0$, which contradicts the choice of $s_1$.
If $\phi(t)<0$ for any $t \geq t_1$, then $\phi'(t)>0$ for $t \geq t_1$. The above arguments also lead to a contradiction. Hence $\theta'(r) = e^{-t} \phi'(t)<0$ for any $r>0$ and $\theta \in (0,\pi)$.
\end{proof}

The polar profile has the following local behaviour near the origin and infinity:
\begin{lemma}\label{la:thetaAsymtot}
Let $\theta$ be a solution of (\ref{eq:thetaODE})-(\ref{eq:thetaBoundary}). Then for $h \gg1$
\begin{equation}\label{eq:theta_infinity}
\theta(r) = \frac{\alpha_h}{\sqrt{r}}\exp(-\sqrt{h}r)+o(\exp(-\sqrt{h}r)) \quad \text{as } r \to +\infty
\end{equation}
where $\alpha_h$ is a parameter depending on $h$, and
\begin{equation}\label{eq:theta_0}
\theta(r)=\pi-\frac{h}{2}r + o(r) \quad \text{as } r \searrow 0.
\end{equation}
\end{lemma}
\begin{proof}
For \eqref{eq:theta_infinity} we adapt the argument for the exponential decay of axisymmetric solution in \cite{strauss1977existence}. The axisymmetric solution is continuous for $r >0$ and it satisfies the equation $\theta''+\frac{1}{r} \theta' = F(\theta,r)$ where 
\[
F(\theta,r)=\frac{1}{r^2} \sin \theta \cos \theta - \frac{2}{r}\sin^2 \theta + h\sin \theta, \quad  r \in (0,\infty).
\]
Hence $\frac{1}{r}(r\theta')'= \theta''+\frac{1}{r} \theta$ is continuous and $\theta$ is a $C^2$ function for $r>0$. Now $v(r)=r^{\frac{1}{2}} \theta(r)$ satisfies the equation
\[
\left( \frac{1}{2}v^2 \right)'' = (v')^2 + v v'' = (v')^2 + \left( q(r)-\frac{1}{4r^2} \right) v^2
\]
where
\[
q(r) = \frac{F(\theta(r))}{\theta(r)} 
= \frac{\sin \theta}{\theta} \left[ \frac{1}{r} \left(\frac{1}{r}\cos \theta-2\sin \theta \right) + h \right].
\]
and $q(r)-\frac{1}{4r^2}>1$ for large enough $r$. Thus $w=v^2$ satisfies the inequality $w'' \geq c w$ for some constant $c>0$. It implies the exponential decay of $w$, hence of $\theta$, i.e.
\[
\theta(r)=\frac{\alpha_h}{\sqrt{r}} \exp(-\beta_h r) + o(\exp(-\beta_h r))
\]
where $\alpha_h ,\beta_h$ are parameters depending on $h$ and $\beta_h>0$. Plugging this ansatz into the \eqref{eq:thetaODE} we get $\beta_h = \sqrt{h}$.

Let us now turn to \eqref{eq:theta_0}. According to the rescaled equation (\ref{eq:rescaled}) we have $\frac{\sin \phi}{\phi} \left( \cos \phi + 2e^{-t} \sin \phi - he^{-2t} \right) \geq \frac{1}{2}$ for $t$ large enough and hence
\[
\left( \frac{1}{2}\phi^2 \right)'' = (\phi')^2 + \phi^2 \bigg[ \frac{\sin \phi}{\phi} \left( \cos \phi + 2e^{-t} \sin \phi - he^{-2t} \right)\bigg] \geq c \phi^2
\]
for some $c>0$, which implies the exponential decay of $\phi$ and therefore
\[
\theta(r)=\pi - \tilde{\alpha}_h r^{\tilde{\beta}_h} + o(r^{\tilde{\beta}_h}) \quad \text{as } r \searrow 0.
\]
where $\tilde{\alpha}_h ,\tilde{\beta}_h$ are parameters depending on $h$ and $\tilde{\beta}_h>0$. Inserting this into (\ref{eq:thetaODE}) yields $\tilde{\alpha}_h = h/2$ and $\tilde{\beta}_h=1$.
\end{proof}

\begin{lemma}
We have the following estimates for $h$ sufficiently large
\begin{equation}
|\cos \theta - r\sin \theta| < \frac{3}{2},
\end{equation}
\begin{equation}
r^2(\theta'(r))^2 \geq \sin^2 \theta(r)
\end{equation}
and
\begin{equation}
h-\frac{3}{2r}\sin \theta \geq 0
\end{equation}
for all $r \in (0,\infty)$. 
\end{lemma}

In fact, as $h \to \infty$, we have $h-\frac{c}{r}\sin \theta \geq 0$ for constants $c$ approaching $2$.

\begin{proof}
We consider the function $f(r)=r\sin(\theta(r))-\cos(\theta(r))$. By Lemma \ref{la:thetaMonoton} there exists a unique $r^* >0$ so that $\theta(r^*)=\frac{\pi}{2}$. First we need an estimate of $r^*$. It follows from Young's inequality that
\begin{align*}
|E_H(\m)|= \bigg| \int_{\R^2} (\m-\ein_3) \cdot (\nabla \times \m) \dif x \bigg| 
& \leq \int_{\R^2} \frac{1}{2} |\nabla \m|^2 + \frac{1}{2} |\m - \ein_3|^2 \dif x
\end{align*}
From the (axisymmetric) upper bound in \cite{melcher2014chiral} we know that $E(\m) < 4\pi$ and $1-\cos \theta >1$ in $(0,r^*)$. Then
\begin{align*}
4\pi > E(\m)  &\geq \int_{\R^2} \frac{1}{2} |\nabla \m|^2 + \frac{h}{2} |\m - \ein_3|^2 \dif x - |E_H(\m)| \\
& \geq \frac{h-1}{2} \int_{\R^2} |\m - \ein_3|^2 \dif x 
= (h-1) \int_{\R^2}  (1-m_3) \dif x \\
&= 2\pi (h-1) \int_0^{\infty} (1-\cos \theta) r \dif r 
> 2\pi (h-1) \int_0^{r^*} r \dif r\\
&=\pi (h-1)(r^*)^2
\end{align*}
and hence
\[
r^* < \frac{2}{\sqrt{h-1}} < \frac{1}{2}
\]
for $h>17$. Now we have $f(r^*)=r^*<1$, $f(r)>-1$ for all $r>0$ and $ f(r) \searrow -1$ as $r \to +\infty$ by \eqref{eq:theta_infinity}. If $f$ is bigger than 1 somewhere in $(r^*,\infty)$, there would exist a $\tilde r \in (r^*,\infty)$ such that $f(\tilde r)>1$, $f'(\tilde r)=0$. But for $h$ sufficiently large we can always deduce $f''(\tilde r)>0$, which leads to a contradiction. Hence $f(r)<1$ in $(r^*,\infty)$ and
\[
|\cos \theta - r \sin \theta| < 1+r^* < \frac{3}{2}.
\]

For the second inequality we consider the function $g(r)=r^2(\theta'(r))^2-\sin^2 \theta(r)$. It is clear that $g(0)=0$ and $\lim_{r \to \infty} g(r) = 0$ by \eqref{eq:theta_infinity}. According to the \eqref{eq:thetaODE} we have
\begin{align*}
g'(r) &= 2r(\theta'(r))^2 + 2r^2 \theta'(r) \theta''(r) - 2 \sin \theta(r) \cos \theta(r) \theta'(r)\\
& = 2r \sin \theta(r) \theta'(r) (hr-2\sin \theta(r)).
\end{align*}
If $g(r)<0$ at some point in $(0,\infty)$, there would exist two points $0<r_1<r_2<\infty$ so that $g'(r_1)=g'(r_2)=0$. But this is impossible since the function $r \mapsto hr-2\sin \theta(r)$ has only one zero-point in $(0, \infty)$ due to the monotonicity of $\theta$.

By \eqref{eq:theta_0} we know that $r \mapsto hr-\frac{3}{2} \sin \theta(r)$ is positiv in the near of 0. If the map is negative at some point, it would have two zero points in $(0,\infty)$, which is impossible due to the monotonicity of $\theta$. Therefore $hr-\frac{3}{2}\sin \theta \geq 0$ and the last inequality follows.
\end{proof}

Using the properties above we deduce the uniqueness.
\begin{lemma}
The solution of (\ref{eq:thetaODE})-(\ref{eq:thetaBoundary}) is unique for $h \gg 1$.
\end{lemma}

\begin{proof}
Suppose there exist two different solutions $\theta_1$ and $\theta_2$. We choose $0< \varepsilon \ll 1$ so that $\theta_1(r_1)=\theta_2(r_2)=\varepsilon$ with $r_1,r_2 \gg 1$. Without loss of generality, we may assume $r_1 > r_2$. Define $\tilde{\theta}_2(r)=\theta_2(r-r_1+r_2)$ for $r \geq r_1$. Then
\[
\tilde{\theta}''_2(r)=-\frac{1}{r-r_1+r_2}\tilde{\theta}_2' + \frac{1}{(r-r_1+r_2)^2} \sin \tilde{\theta}_2 \cos \tilde{\theta}_2 - \frac{2}{r-r_1+r_2} \sin^2 \tilde{\theta}_2 + h\sin \tilde{\theta}_2.
\]
If $\tilde{\theta}'_2(r_1) > \theta'_1(r_1)$, then by $\tilde{\theta}_2(r_1)=\theta_1(r_1)$ and the Taylor's formula we would have $\tilde{\theta}_2(r)> \theta_1(r)$ for $r>r_1$ and very close to $r_1$. Using the decay property (\ref{eq:theta_infinity}) for $r \gg 1$  and the monotonicity of trigonometric functions on $(0,\frac{\pi}{4})$ we deduce 
\[
\tilde{\theta}_2''(r)-\theta_1''(r) = h(\sin \tilde{\theta}_2 - \sin \theta_1) + o(\sin \tilde{\theta}_2) > 0.
\]
Thus we obtain $\tilde{\theta}_2(r)>\theta_1(r)$ for $r >r_1$ and $\tilde{\theta}'_2(r)-\theta'_1(r)$ is strictly increasing. It contradicts the fact $\tilde{\theta}'_2 \to 0$ and $\theta'_1 \to 0$ as $r \to \infty$. For $\tilde{\theta}'_2(r_1) \leq \theta'_1(r_1)$ a similar argument applies.
\end{proof}

\nocite{*}
\bibliographystyle{abbrv}
\bibliography{Stability_and_Traveling_waves}

\end{document}